\documentclass[a4paper,11pt,twoside]{article}
\usepackage{pstricks,pstricks-add,pst-math,pst-xkey}

\usepackage{pictexwd}

\date{}
\usepackage{amsmath}
\usepackage{amssymb}
\usepackage{amsthm}

\newcommand{\C}{\mathbb{C}}
\newcommand{\R}{\mathbb{R}}
\newcommand{\g}{\mathfrak{g}}

\newcommand{\p}{\mathfrak{p}}
\newcommand{\bdm}{\begin{displaymath}}
\newcommand{\edm}{\end{displaymath}}

\theoremstyle{definition}

\newtheorem{lem}{Lemma}[section]
\newtheorem{thm}{Theorem}[section]

\newtheorem{prop}{Proposition}[section]
\newtheorem{rem}{Remark}[section]

\title{Harmonic maps of finite uniton number  into $G_2$}
\author{N. Correia and R. Pacheco}

\begin{document}

\maketitle

\begin{abstract}
We establish explicit formulae for canonical factorizations of extended solutions corresponding to harmonic maps of finite uniton number into the exceptional Lie group $G_2$  in terms of the Grassmannian model for  the group of based algebraic loops in $G_2$. A description of the ``Frenet frame data" for such harmonic maps is given. In particular, we show that harmonic spheres into $G_2$ correspond to solutions of  certain algebraic systems of quadratic and cubic equations.
\end{abstract}

\section{Introduction}
In the seminal paper \cite{U}, Uhlenbeck observed that harmonic maps from a Riemann surface into a Lie group $G$ correspond to certain holomorphic maps, the \emph{extended solutions}, into the loop group
$$\Omega G=\{\gamma:S^1\to G \,(\mbox{smooth})\mid\,\gamma(1)=e\}.$$
When the Fourier series associated to an extended solution has finitely many terms, the corresponding harmonic map is said to have \emph{finite uniton number}.  Such harmonic maps can be obtained from a constant by applying a finite number of ``B\"{a}cklund-type transforms''. In terms of loop groups, this means that extended solutions corresponding to harmonic maps of finite uniton number into the unitary group admit  factorizations into linear factors. Subsequently, Burstall and Guest \cite{BG} generalized this to other Lie groups
and gave ``Weierstrass-type formulae'', by means of which harmonic maps of finite uniton number can be described in terms of certain meromorphic functions on $M$. This was accomplished by
using a method inspired by the Morse theoretic interpretation of the Bruhat decomposition of the loop subgroup
$$\Omega_\mathrm{alg}G=\{\gamma\in\Omega G \,|\,\mbox{$\gamma$ and $\gamma^{-1}$ have finite Fourier series}\}.$$
More precisely:

Consider the energy functional $E:\Omega G\to \mathbb{R}$ given by $E(\gamma)=\int_{S^1}|\gamma'|^2$. This is a Morse-Bott function on the K\"{a}hler manifold $\Omega G$ and its critical manifolds are precisely the conjugacy classes of homomorphisms $S^1\to G$. If $\Omega_\xi$ is such a class and $U_\xi$ is the unstable manifold of $\Omega_\xi$ with respect to the flow of the gradient vector field $-\nabla E$, then the Bruhat decomposition corresponds to the decomposition $\Omega_\mathrm{alg}G=\bigcup_\xi U_\xi$. Moreover, $U_\xi$ carries a structure of vector bundle over $\Omega_\xi$ and, given a finite uniton number harmonic map $\varphi:M\to G$, it can be proven that it admits an extended solution $\Phi:M\to \Omega_\mathrm{alg}G$ which takes values in some $U_\xi$ off a discrete subset $D$ of $M$. The unstable manifolds $U_\xi$ admit a suitable Lie theoretic description, which can therefore be applied to the study of harmonic maps.

In the present paper, we  explore further this point of view. For certain pairs of elements $\xi,\xi'$, we introduce natural holomorphic fibre bundle morphisms $\mathcal{U}_{\xi,\xi'}:U_\xi\to U_{\xi'}$, which we use later to construct canonical factorizations for loops in $\Omega_\mathrm{alg}G_2$. These morphisms transform extended solutions in new extended solutions, that is, if $\Phi:M\setminus D \to U_\xi$ is an extended solution, then $\mathcal{U}_{\xi,\xi'}\circ\Phi:M\setminus D \to U_{\xi'}$ is a new extended solution. Hence the canonical factorizations for loops induce canonical factorizations for extended solutions. In the case of the exceptional Lie group $G_2$, our factorizations are finer than those constructed in \cite{BG}. Explicit formulae for these factorizations are given in terms of the Grassmannian model for $\Omega_\mathrm{alg}G_2$.

The Grassmannian model for loop groups  was exploited for the first time in the study of harmonic maps by Segal \cite{S}. In this setting, as observed by Guest \cite{G}, the ``holomorphic data" producing  harmonic maps can be organized in terms of ``Frenet frames". More recently, the Grassmannian model was used by the second author \cite{Pa} and Svensson and Wood \cite{SW} in the study of harmonic maps into classical Lie groups and their inner symmetric spaces. 

In this paper,  we give a description of the ``Frenet frame data" for harmonic maps of finite uniton number into $G_2$ and its inner symmetric space $G_2/SO(4)$, the Grassmannian of associative $3$-planes. In particular, we show that each harmonic sphere into $G_2$ can be obtained  from  a solution of a certain algebraic system of quadratic and cubic equations.

\section{The fundamental representation of $G_2$ }\label{rep}
We start by reviewing  some aspects concerning the fundamental representation of $G_2$. For more details we refer the reader to \cite{Ful}, Lecture 22.

It is well known that the exceptional compact simple Lie group $G_2$ is exactly the automorphism
group of $\mathbb{O}$, the  real
$8$-dimensional division algebra of octonions.
Equip $\mathbb{O}$ with the natural inner product
$\langle x,y \rangle_\R ={\rm{Re}}(x\cdot \bar{y})=\frac12(x\cdot
\bar{y}+y\cdot\bar{x})$. Since this metric is defined as the
multiplication, we get $G_2\subset SO(\mathbb{O})$. Every automorphism of the octonions fixes the subspace
$\mathbb{R}\cdot 1\subset \mathbb{O}$ and thus preserves the subspace of octonions orthogonal to the identity, that is, the $7$-dimensional subspace consisting of all imaginary octonions, ${\rm{Im}}(\mathbb{O})$. So, if we
identify ${\rm{Im}}(\mathbb{O})=\mathbb{R}^7$, we get exactly the
fundamental representation $G_2\subset SO(7)$, which is the smallest non-trivial representation of $G_2$. If we fix a maximal torus $T\subset G_2$, which is known to be $2$-dimensional, and a Weyl chamber $\mathcal{W}$ in $\mathfrak{t}$, the Lie algebra of $T$, the corresponding weight diagram looks like:
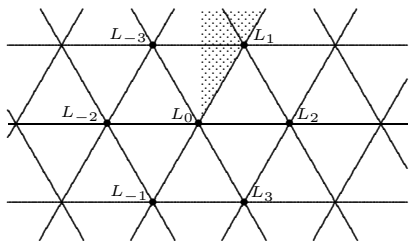
\begin{center}
\quad\!\!\!
\beginpicture
\setcoordinatesystem units <1.2cm,1.2cm>
\put {\phantom{.}} at -2.09110 -1.28401
\put {\phantom{.}} at -2.09110 1.27012
\put {\phantom{.}} at 2.28741 -1.28401
\put {\phantom{.}} at 2.28741 1.27012
\setlinear
\setshadesymbol <z,z,z,z> ({\fiverm .})
\setshadegrid span <1.3600000pt>
\vshade 0.000000 0.000000 0.000000
 0.005329 0.009230 1.270090
 0.733287 1.270090 1.270090
 /
\plot -0.73180 1.26751 0.73982 -1.28140 /
\plot -0.73982 -1.28140 0.73180 1.26751 /
\putrule from -2.08849 0.00000 to 2.28480 0.00000
\plot -2.08849 -0.86603 2.28480 -0.86603 /
\plot 2.28480 0.86603 -2.08849 0.86603 /
\plot -1.73982 -1.28140 -0.26820 1.26751 /
\plot 0.26820 1.26751 1.73982 -1.28140 /
\plot 1.73180 1.26751 0.26018 -1.28140 /
\plot -1.73180 1.26751 -0.26018 -1.28140 /
\plot -1.26820 1.26751 -2.08849 -0.15327 /
\plot -2.08849 0.15327 -1.26018 -1.28140 /
\plot 2.28480 0.49329 1.26018 -1.28140 /
\plot 1.26820 1.26751 2.28480 -0.49329 /
\put {$\scriptscriptstyle\bullet$} at 0.00000 0.00000
\put {$\scriptscriptstyle\bullet$} at 1.00000 0.00000
\put {$\scriptscriptstyle\bullet$} at -1.00000 0.00000
\put {$\scriptscriptstyle\bullet$} at 0.50000 0.86603
\put {$\scriptscriptstyle\bullet$} at -0.50000 -0.86603
\put {$\scriptscriptstyle\bullet$} at -0.50000 0.86603
\put {${\scriptscriptstyle\bullet}$} at 0.50000 -0.86603
\put {\tiny{$L_1$}} at 0.72 0.97
\put {\tiny{$L_2$}} at 1.2 0.1
\put {\tiny{$L_3$}} at 0.7 -0.77
\put {\tiny{$L_{-1}$}} at -0.76 -0.77
\put {\tiny{$L_{-2}$}} at -1.3 0.1
\put {\tiny{$L_{-3}$}} at -0.77 0.97
\put {\tiny{$L_0$}} at -0.18 0.1
\endpicture
\end{center}

 The
octonionic-imaginary part of the product of $x,y\in\mathbb{C}^7={\rm{Im}}(\mathbb{O})\otimes\mathbb{C}$  will be denoted by
 $x\cdot y\in\mathbb{C}^7$and consider the orthogonal decomposition of $\C^7$ into
one-dimensional weight subspaces:
$\C^7=\bigoplus_{i=-3}^{3}L_i,$
where  $L_{-i}=\overline{L_i}$.
Let $\omega_j$ be the weight of $L_j$. Clearly, $L_i\cdot L_j$ is the weight space for $\omega_i+\omega_j$, if this is a weight, and zero otherwise. Hence, we obtain from the weight diagram the following octonionic multiplication table:
$$\begin{tabular}{|c||c|c|c|c|c|c|c|}
  \hline
  $\cdot$ & $L_0$ & $L_1$ & $L_2$ & $L_3$ &  $\overline{L_1}$ &  $\overline{L_2}$  &  $\overline{L_3}$  \\
  \hline \hline
  $L_0$ & $0$ & $L_1$ & $L_2$ & $L_3$ & $\overline{L_1}$ & $\overline{L_2}$ & $\overline{L_3}$   \\\hline
  $L_1$ & $L_1$ & $0$ & $0$ & $L_2$ & $L_0$ & $\overline{L_3}$ & $0$ \\\hline
  $L_2$ & $L_2$ & $0$ & $0$ & $0$ & $L_3$ & $L_0$ & $L_1 $  \\\hline
 $L_3$ & $L_3$ & $L_2$ & $0$ & $0$ & $0$ & $\overline{L_1}$ & $L_0$ \\\hline
 $\overline{L_1}$& $\overline{L_1}$ & $L_0$ & $L_3$ & $0$ & $0$ & $0$ & $\overline{L_2}$  \\\hline
  $\overline{L_2}$ & $\overline{L_2}$ & $\overline{L_3}$ & $L_0$ & $\overline{L_1}$ & $0$ & $0$ & $0$  \\\hline
 $\overline{L_3}$& $\overline{L_3}$& $0$ & $L_1$ & $L_0$ & $\overline{L_2}$ & $0$ & $0$   \\
  \hline
\end{tabular}$$

The  positive roots associated to the pair $(T,\mathcal{W})$ are given by:
\begin{equation}\label{roots}
\alpha_1,\alpha_2,\alpha_1+\alpha_2, 2\alpha_1+\alpha_2, 3\alpha_1+\alpha_2,3\alpha_1+2\alpha_2,
 \end{equation}
 with the simple roots $\alpha_1,\alpha_2$ dual to
 the elements $H_1,H_2\in\mathfrak{t}$
defined by
\begin{equation}\label{Hs}
 H_1=\left\{\begin{array}{ll}
        2\sqrt{-1} & \mbox{on $L_1$} \\
        \sqrt{-1} &  \mbox{on $L_2$}\\
        \sqrt{-1} & \mbox{on $\overline{L}_3$}
      \end{array}\right.\quad\quad H_2=\left\{\begin{array}{ll}
        \sqrt{-1} & \mbox{on $L_1$} \\
        0 & \mbox{on $L_2$}\\
        \sqrt{-1} &  \mbox{on $\overline{L}_3$}
      \end{array}\right.,\end{equation}
in the following sense: $\alpha_i(H_j)=\sqrt{-1}\delta_{ij}$.
The weights of the fundamental representation can be written in terms of the simple roots as follows:
$\omega_1=2\alpha_1+\alpha_2$, $\omega_2=\alpha_1$, $\omega_{-3}=\alpha_1+\alpha_2$.

Given an isotropic subspace $\mathcal{D}\subset \mathbb{C}^7$, we denote by $\mathcal{D}_0$ its stabilizer and by $\mathcal{D}^a$ its annihilator:
$$\mathcal{D}_0=\big\{x\in\C^7|\,\, x\cdot \mathcal{D}\subseteq \mathcal{D} \big\},\quad \mathcal{D}^a=\big\{x\in\C^7|\,\, x\cdot \mathcal{D}=0 \big\}.$$ We have:
\begin{lem}
\emph{Let $\mathcal{D}\subset \C^7$ be an isotropic one-dimensional subspace. Then $\dim {D}^a=3$ and $\mathcal{D}_0=\overline{\mathcal{D}^a}^\perp$.}
\end{lem}
\begin{proof}
  Since $G_2$ acts
transitively on the one-dimensional isotropic subspaces, we can
take ${\mathcal{D}}=L_1$.
From the octonionic multiplication table we see that
$\mathcal{D}_0=L_0\oplus L_1\oplus L_2\oplus \overline{L}_3$ and $\mathcal{D}^a= L_1\oplus L_2\oplus \overline{L}_3$. Hence  $\dim {\mathcal{D}}^a=3$ and $\mathcal{D}_0=\overline{\mathcal{D}^a}^\perp$.
  \end{proof}

   An isotropic $2$-plane $\mathcal{D}$ such that $\mathcal{D}\cdot\mathcal{D}=0$ is called an\emph{ complex coassociative $2$-plane}. If $\mathcal{D}$  is  an complex coassociative $2$-plane, then we have an orthogonal decomposition:
   $\C^7=\mathcal{D}\oplus \mathcal{A}\oplus \overline{\mathcal{D}},$
   where $\mathcal{A}=( \mathcal{D}\oplus\overline{\mathcal{D}})^\perp$ is called a \emph{complex associative $3$-plane}.

\section{Grassmannian model for loop groups}

Fix on $\mathbb{C}^{n}$ the standard complex inner product $\langle
\cdot,\cdot\rangle$ and let $e_{1},\ldots,e_{n}$ be the standard basis vectors for
 $\mathbb{C}^{n}$. Given a complex subspace $V\subset \C^n$, we denote by $\pi_V$ the orthogonal projection onto $V$.
Let $H$ be the Hilbert space
 of square-summable $\C^{n}$-valued
functions on the
 circle and $ \langle\cdot,\cdot\rangle_H$ the induced complex inner product. This is the closed space generated by
the functions
$\lambda\mapsto\lambda^{i}e_{j}$, with $i\in\mathbb{Z}$ and $j=1,\ldots,n$: $$H=\mbox{Span}\{ \lambda^{i}e_{j}\,|\,i\in\mathbb{Z},\,
 j=1,\ldots,n\}.$$
 Consider the closed subspace $H_+$ of $H$ defined by
 $$H_+=\mbox{Span}\{ \lambda^{i}e_{j}\,|\,i\geq 0,\,
 j=1,\ldots,n\}.$$
 Let $\mbox{\emph{Grass}}(H)$ denote the set of all closed vector
 subspaces $W\subset H$ such that:
 the projection map $W\rightarrow H_+$ is
 Fredholm, and the projection map $W\rightarrow
 H_+^\perp$ is Hilbert-Schmidt;
the images of the projections maps $W^{\perp}\rightarrow H_+$, $W\rightarrow
 H_+^\perp$ are contained in $C^{\infty}(S^{1};\mathbb{C}^{n})$.  Define
$${Gr}=\{W\in \mbox{\emph{Grass}}(H)\,|\, \lambda W\subseteq
W \}.$$

    Pressley and Segal \cite{PS} showed that the action of the infinite-dimensional Lie group
 $$\Lambda U(n)=\big\{\gamma:S^1\to U(n)\,|\, \mbox{$\gamma$ is smooth}\big\}$$
 on $Gr$ defined by $\gamma  W=\{\gamma f\,|\,f\in W\}$ is transitive. By considering Fourier series, it is easy to see that the
 isotropy subgroup at $H_+$ is precisely
 ${U}(n)$. Hence
 $$Gr\cong \Lambda U(n)/U(n)\cong \Omega U(n).$$  This homogeneous space carries a natural invariant structure of K\"{a}hler manifold \cite{PS}.
\begin{rem}\label{rem}
 Given $W\in {Gr}$, it is known \cite{PS} that
$\dim W \ominus \lambda W=n$, where $W \ominus
\lambda W$ denotes the orthogonal complement of $\lambda W$ in
$W$. If we choose an orthonormal basis for $W\ominus\lambda W$,
$\{w_1,\ldots,w_n\}$, we can put the vector-valued functions
$w_i$ side by side to form an $(n\times n)$-matrix valued
function $\gamma$  on $S^1$, that is, a loop $\gamma\in\Lambda
{U}(n)$. It can be shown \cite{PS} that $W=\gamma  H_+$.
\end{rem}

If $G$ is a subgroup of $U(n)$, we shall denote by $Gr(G)$ the subspace of $Gr$ that corresponds to $\Omega G$.

\subsection{Grassmannian model for $\Omega SO(n)$}
We consider the special orthogonal group  $SO(n)$ as a subgroup of $U(n)$. For each $X$ in $\C^n$ denote
by $\overline{X}$ its complex conjugate. The Grassmannian model of $\Omega SO(n)$ is given by:
\begin{prop}\cite{PS}
\textit{A subspace $W\in {Gr}$ corresponds to a loop in
${SO}(n)$ if, and only if, it belongs to
$${Gr}\big(SO(n)\big)=\big\{{W\in \mbox{Gr}\,|\,
\overline{W}^{\perp}=\lambda W \big\}}.$$}
\end{prop}

\subsection{Grassmannian model for $\Omega G_2$}

Take the complex bilinear extension to $\C^7$ of the octonionic product on $\mathbb{R}^7$ and use it to define a product on the Hilbert space $H$ of square-summable $\C^7$-valued functions on the circle: if $f,g\in H$, then $(f\cdot g)(\lambda)=f(\lambda)\cdot g(\lambda)$.

The Grassmannian model of $\Omega G_2$ is given by:
\begin{prop}
\emph{A subspace $W\in {Gr}\big(SO(7)\big)$ corresponds to
a loop in $G_2$ if, and only if, it belongs to \bdm
{Gr}(G_2)=\{W\in {Gr}\big(SO(7)\big)\,|\,
\,W^{sm}\cdot W^{sm}\subseteq W^{sm}\}\edm(here $W^{sm}$ denotes the
subspace  of smooth functions in $W$, which is dense \cite{PS})}.
\end{prop}
\begin{proof}
The proof of Theorem 8.6.2 in \cite{PS} can be adapted to this case.
If $\gamma\in \Omega G_2$, then it is clear that $\gamma H_+$
belongs to $Gr(G_2)$, since $G_2$ acts on $\mathbb{C}^7$ by
automorphisms. Conversely, suppose that $W^{sm}\cdot W^{sm}\subseteq
W^{sm}$. Since $\overline{W}^\perp =\lambda W$, we have $W\ominus
\lambda W=W\cap \overline{W}$. On the other hand, we know  that $\dim  W\ominus
\lambda W =7$ and $W\ominus
\lambda W$  consists of smooth functions. Hence $W\cap \overline{W}$ is a $7$-dimensional subalgebra of $H$ with respect to the product induced by the octonionic product on $\mathbb{R}^7$; consequently, for any $\lambda \in S^1$, the evaluation map
 at $\lambda$, ${\rm{ev}}_\lambda:W\cap
\overline{W}\to \mathbb{C}^7$, defines an isomorphism:
${\rm{ev}}_\lambda(\alpha\cdot\beta)=\alpha(\lambda)\cdot
\beta(\lambda)$. Set $\gamma(\lambda)={\rm{ev}}_\lambda\circ {\rm{ev}}^{-1}_1$. Since   ${\rm{ev}}_\lambda$ commutes with
complex conjugation and, as we have seen, ${\rm{ev}}_\lambda$ is an isomorphism for any $\lambda \in S^1$, then the
loop $\gamma$ belongs to $\Omega G_2$. By Remark \ref{rem}, we
have  $W=\gamma H_+$.
\end{proof}
\subsection{The algebraic Grassmannian}
A loop
 $\gamma \in\Omega U(n)$ is said to be algebraic if both $\gamma$ and $\gamma^{-1}$ have finite
Fourier series. Denote by
$\Omega_{\rm{alg}}U(n)$ the subgroup of algebraic loops. This
subgroup acts on
$${Gr}_{\mathrm{alg}}=\{W \in {Gr}:\,
\lambda^k H_+ \subseteq W\subseteq \lambda^{-k}
H_+\,\,\textrm{for some } k\in\mathbb{N} \},
$$
and we have  ${Gr}_{\mathrm{alg}}\cong
 \Omega_{\mathrm{alg}}{U}(n)$ (see \cite{PS} for details).

 \section{The Bruhat Decomposition of ${Gr}_{\mathrm{alg}}(G)$}

Consider now a compact matrix semi-simple Lie group $G$. Fix a maximal torus $T$ of $G$ with Lie algebra $\mathfrak{t}\subset \mathfrak{g}$, let $\Delta\subset \sqrt{-1} \mathfrak{t}^*$ be the corresponding
set of roots and, for any $\alpha\in \Delta$, denote by $\g_\alpha$ the corresponding root space. The integer lattice $I =
(2\pi)^{-1} \exp^{-1}(e)\cap \mathfrak{t}$ may be identified with the group of homomorphisms $S^1\to  T$, by
associating to $\xi\in I$ the homomorphism
$\gamma_\xi$ defined by $\gamma_\xi(\lambda)=\exp{(-\sqrt{-1}\ln(\lambda)\xi)}$. Denote by $\g^\xi_i$ the $\sqrt{-1}\,i$-eigenspace of $\mathrm{ad}{\xi}$, with $i\in\mathbb{Z}$.
We have on $\mathfrak{g}^\C$ a structure of graded Lie algebra:
\begin{equation*}
\g^\C=\!\!\!\bigoplus_{i\in\{-r(\xi)\ldots,r(\xi)\}}\!\!\!\mathfrak{g}^\xi_i,\quad [\mathfrak{g}^\xi_i, \mathfrak{g}^\xi_j]\subset \mathfrak{g}^\xi_{i+j},
\end{equation*}
 where $r(\xi)=\mathrm{max}\{i\,|\,\,\g_i^\xi\neq 0\}$. Moreover:
 \begin{equation}\label{gis}
\g_i^\xi=\!\!\bigoplus_{\alpha(\xi)=\sqrt{-1}\,i}\!\!\g_\alpha.
\end{equation}

The adjoint action of $\gamma_\xi$ on $\g^\C$  is given by:
\begin{lem}\label{adj}
 \emph{ For each $X_j\in \g^\xi_j$,
  $\gamma_\xi X_j\gamma_\xi^{-1}=\lambda^jX_j.$}
\end{lem}
\begin{proof}
  Taking account the well known formula $\mathrm{Ad}\big(\exp(\eta)\big)=\exp\big(\mathrm{ad}(\eta)\big)$, for all $\eta\in\g^\C$, we have:
  \begin{align*}
    \gamma_\xi X_j\gamma_\xi^{-1}&= \exp{\big(-\sqrt{-1}\ln(\lambda)\xi\big)}X_j\exp{\big(\sqrt{-1}\ln(\lambda)\xi\big)}=\exp\big(-\sqrt{-1}\ln(\lambda)\mathrm{ad}(\xi)\big)X_j\\
    &=\sum_{n\geq 0} \frac{\big(-\sqrt{-1}\ln(\lambda)\big)^n}{n!}(\mathrm{ad}\,\xi)^nX_j=\sum_{n\geq 0} \frac{\big(j\ln(\lambda)\big)^n}{n!}X_j=\lambda^jX_j.
  \end{align*}
\end{proof}

Set $\Lambda^+G^\C=\{\gamma:S^1\to G^\C\,|\,\,\mbox{$\gamma$ extends holomorphically for $|\lambda|\leq 1$}\}$.
For each
$\xi\in I$, we write
$\Omega_\xi=\{g\gamma_\xi g^{-1} \,|\,\, g \in G\},$ the conjugacy class of homomorphisms $S^1\to G$ which contains $\gamma_\xi$.
This is a complex homogeneous space \cite{BG}:
$$\Omega_\xi\cong G^\C\big/P_\xi,\,\mbox{with}\,P_\xi=G^\C\cap \gamma_\xi\Lambda^+G^\C\gamma_\xi^{-1}.$$
Taking account Lemma \ref{adj}, one can easily check  that the Lie algebra of the isotropic subgroup $P_\xi=G^\C\cap \gamma_\xi\Lambda^+G^\C\gamma_\xi^{-1}$ is precisely the parabolic subalgebra induced by $\xi$:
$\mathfrak{p}_\xi=\bigoplus_{i\leq 0}\g^\xi_i$.

Now, choose a fundamental Weyl  chamber $\mathcal{W}$  in $\mathfrak{t}$. The intersection of $I$ with this will be
denoted by $I'$. We have:
\begin{thm}\cite{PS}
 \emph{Bruhat decomposition: }$Gr_\mathrm{alg}(G)=\bigcup_{\xi\in I'}\Lambda^+_\mathrm{alg}G^\C\gamma_\xi H_+$.
\end{thm}

Define $U_\xi\subset \Omega_{\mathrm{alg}}G$ by  $U_\xi H_+=\Lambda^+_{\mathrm{alg}}G^\C\gamma_\xi H_+.$ This is a complex homogeneous space of the
 group $\Lambda^+_{\mathrm{alg}}G^\C$, and the isotropy subgroup at $\gamma_\xi$ is the subgroup
$\Lambda^+_{\mathrm{alg}}G^\C\cap \gamma_\xi \Lambda^+G^\C \gamma_\xi^{-1}.$ Moreover,  $U_\xi$ carries a structure of holomorphic vector bundle over $\Omega_\xi$ and the bundle map
$u_\xi:U_\xi\to \Omega_\xi$ is precisely the natural map
\begin{equation*}
\Lambda^+_{\mathrm{alg}}G^\C\Big/ \Lambda^+_{\mathrm{alg}}G^\C\cap \gamma_\xi \Lambda^+G^\C \gamma_\xi^{-1}\to G^\C \big/ P_\xi
\end{equation*}
given by  $[\gamma]\mapsto [\gamma(0)]$ (see \cite{BG} for details).

\begin{rem} As explained  in \cite{PS} (see also \cite{BG}), the Bruhat decomposition admits a nice Morse theoretic approach. Consider the usual energy functional on paths $E:\Omega G\to \mathbb{R}$. This is a Morse-Bott function and its critical manifolds are precisely the
conjugacy classes of homomorphisms $S^1\to G$. For each $\xi\in I$, $U_\xi$ is the unstable manifold of $\Omega_\xi$
with respect to the flow of the gradient vector field $-\nabla E$ defined by the natural K\"{a}hler structure on $\Omega G$, and each $\gamma\in U_\xi$ flows to the homomorphism $u_\xi(\gamma)$.

\end{rem}

Take $\gamma\in U_\xi\subset  \Omega_{\mathrm{alg}}G$ and $W=\gamma H_+\in Gr_{\mathrm{alg}}(G)$, with $\lambda^rH_+\subset W\subset\lambda^{-s}H_+$.
Fix $\Psi\in \Lambda^+_{\mathrm{alg}}G^\C$ such that
$W=\Psi\gamma_\xi H_+$. Write
$$\gamma_\xi H_+=\lambda^{-s}A^\xi_{-s}+\ldots+\lambda^{r-1}A^\xi_{r-1}+\lambda^rH_+,$$ where the subspaces $A^\xi_i$ define a flag
$$\{0\}=A^\xi_{-s-1}\subset A^\xi_{-s}\subseteq A^\xi_{-s+1}\subseteq \ldots\subseteq A^\xi_{r-1}\subset A^\xi_r=\C^n.$$
In terms of the grassmannian model, the  bundle map $u_\xi:U_\xi\to\Omega_\xi$ can be described as follows:
 \begin{equation}\label{popo1}
u_\xi(W)=\lambda^{-s}A_{-s}+\ldots+\lambda^{r-1}A_{r-1}+\lambda^rH_+,
\end{equation}
 with $$A_i=\Psi(0)A^\xi_i= p_i(W\cap\lambda^iH_+),$$ where $p_i: H \to\C^n$ is defined by $p_i(\sum\lambda^ja_j)=a_i$.

\begin{rem}
Consider the mutually orthogonal subspaces $E_{-s},\ldots,E_r$ defined by
$E_i={A^\xi_i}\cap {{A^{\xi}}_{i-1}^{\perp}}$. For each $v\in E_i$, we have $\xi v =\sqrt{-1} iv$. Hence,
\begin{equation}\label{ioio}
\g^\xi_i=\!\!\!\bigoplus_{j\in\{-s,\ldots,r\}}\!\!\!\mathrm{Hom}(E_j,E_{j+i})\cap \g^\C.
\end{equation}
\end{rem}

\subsection{Canonical fibre-bundle holomorphic morphisms between unstable manifolds}
In this section we  show that, under some conditions on a pair of elements $\xi,\xi'\in I$, one can construct a canonical holomorphic morphism between the fibre bundles $U_\xi\to\Omega_\xi$ and $U_{\xi'}\to\Omega_{\xi'}$. Later we shall use these morphisms to construct canonical factorizations of algebraic loops and harmonic maps.

Define a partial order over $I$ by:
\begin{equation}\label{gs}
\xi\preceq \xi'\,\,\,\,\,\,\,\, \mbox{if}\,\,\,\,\,\,\,\, \mathfrak{p}^{\xi}_i\subset \mathfrak{p}^{\xi'}_i,
\end{equation}
for all $i\geq 0$, where $\mathfrak{p}_i^\xi=\bigoplus_{j\leq i}\g_j^\xi$.  Since $\xi$ and $\xi'$ are simultaneously diagonalizable, this condition is equivalent to
\begin{equation}\label{gs1}
\g^\xi_j=\bigoplus_{0\leq k\leq j} \g^\xi_j\cap\g^{\xi'}_k,
\end{equation}
for all $j\geq 0$.

\begin{lem}\label{lema3}\emph{Take two elements $\xi,\xi'\in I$ such that $\xi\preceq \xi'$. Then
   $$\Lambda^+_{\mathrm{alg}}G^\C\cap \gamma_\xi \Lambda^+G^\C \gamma_\xi^{-1}\subset \Lambda^+_{\mathrm{alg}}G^\C\cap \gamma_{\xi'} \Lambda^+G^\C \gamma_{\xi'}^{-1}.$$}
\end{lem}
\begin{proof} Since these subgroups are connected, it is sufficient to prove the inclusion at the Lie algebra level.
Now, Lemma \ref{adj} yields
\begin{equation}\label{lal}
\Lambda^+_{\mathrm{alg}}\g^\C\cap \gamma_\xi \Lambda^+\g^\C \gamma_\xi^{-1}=\bigoplus_{i\geq 0 }\lambda^i\mathfrak{p}^\xi_i.
\end{equation}
From (\ref{gs}) and (\ref{lal}) we conclude that
$$\Lambda^+_{\mathrm{alg}}\g^\C\cap \gamma_\xi \Lambda^+\g^\C \gamma_\xi^{-1}\subset \Lambda^+_{\mathrm{alg}}\g^\C\cap \gamma_{\xi'} \Lambda^+\g^\C \gamma_{\xi'}^{-1}.$$
\end{proof}
This lemma allows us to define a $\Lambda^+_{\mathrm{alg}}G^\C$-invariant fibre bundle morphism
$\mathcal{U}_{\xi,\xi'}:U_\xi\to U_{\xi'}$  by
\begin{equation}\label{uxi}
\mathcal{U}_{\xi,\xi'}(\Psi\gamma_{\xi}H_+)=\Psi\gamma_{\xi'}H_+, \quad \Psi\in\Lambda^+_{\mathrm{alg}}G^\C,\end{equation}
whenever  $\xi\preceq \xi'$.
Since the holomorphic structures on  $U_\xi$ and $U_{\xi'}$ are induced by the holomorphic structure on $\Lambda^+_{\mathrm{alg}}G^\C$, the fibre-bundle morphism  $\mathcal{U}_{\xi,\xi'}$ is holomorphic. Moreover:

For each $\xi\in I$ define the subbundle
\begin{equation}\label{h10}
H^{1,0}U_\xi\cong \Lambda^+_{\mathrm{alg}}G^\C\times_{\Lambda^+_{\mathrm{alg}}G^\C\cap \gamma_\xi \Lambda^+G^\C \gamma_\xi^{-1}}
\Lambda^+_{\mathrm{alg}}\g^\C\cap \lambda^{-1} \gamma_\xi \Lambda^+\g^\C \gamma_\xi^{-1}\Big/ \Lambda^+_{\mathrm{alg}}\g^\C\cap \gamma_\xi \Lambda^+\g^\C \gamma_\xi^{-1}
\end{equation}
of the holomorphic tangent bundle
\begin{equation}\label{t10}
T^{1,0}U_\xi\cong \Lambda^+_{\mathrm{alg}}G^\C\times_{\Lambda^+_{\mathrm{alg}}G^\C\cap \gamma_\xi \Lambda^+G^\C \gamma_\xi^{-1}} \Lambda^+_{\mathrm{alg}}\g^\C\Big/  \Lambda^+_{\mathrm{alg}}\g^\C\cap \gamma_\xi \Lambda^+\g^\C \gamma_\xi^{-1}.
\end{equation}
We have:
\begin{lem}
\emph{If $\xi\preceq \xi'$, then $\mathcal{U}_{\xi,\xi'}$ is \emph{super-horizontal}, that is, $D\mathcal{U}_{\xi,\xi'}(H^{1,0}U_\xi)\subset H^{1,0}U_{\xi'}$. }
\end{lem}
\begin{proof}
 Start to observe that
 $$\Lambda^+_{\mathrm{alg}}\g^\C\cap \lambda^{-1} \gamma_\xi \Lambda^+\g^\C \gamma_\xi^{-1}=\bigoplus_{i\geq 0}\lambda^i \p^\xi_{i+1}.$$
 Hence, by (\ref{gs}),
 \begin{equation}\label{d}
   \Lambda^+_{\mathrm{alg}}\g^\C\cap \lambda^{-1} \gamma_\xi \Lambda^+\g^\C \gamma_\xi^{-1}\subset\Lambda^+_{\mathrm{alg}}\g^\C\cap \lambda^{-1} \gamma_{\xi'} \Lambda^+\g^\C \gamma_{\xi'}^{-1}.
 \end{equation}
 On the other hand, the derivative  $D\mathcal{U}_{\xi,\xi'}:T^{1,0}U_\xi\to T^{1,0}U_{\xi'} $, which  corresponds to  a map
$$\Lambda^+_{\mathrm{alg}}\g^\C\Big/  \Lambda^+_{\mathrm{alg}}\g^\C\cap \gamma_\xi \Lambda^+\g^\C \gamma_\xi^{-1}\to
\Lambda^+_{\mathrm{alg}}\g^\C\Big/  \Lambda^+_{\mathrm{alg}}\g^\C\cap \gamma_{\xi'} \Lambda^+\g^\C \gamma_{\xi'}^{-1},$$
is simply given by $[\eta]\mapsto [\eta]$, and  the lemma follows from inclusion (\ref{d}).
\end{proof}

Similarly, the derivative $Du_\xi:T^{1,0}U_\xi\to T^{1,0}\Omega_\xi\subset T^{1,0}U_\xi$ corresponds to a map
 $$\Lambda^+_{\mathrm{alg}}\g^\C\Big/  \Lambda^+_{\mathrm{alg}}\g^\C\cap \gamma_\xi \Lambda^+\g^\C \gamma_\xi^{-1}\to \g^\C/\mathfrak{p}_\xi\hookrightarrow \Lambda^+_{\mathrm{alg}}\g^\C\Big/  \Lambda^+_{\mathrm{alg}}\g^\C\cap \gamma_\xi \Lambda^+\g^\C \gamma_\xi^{-1}$$ and   $u_\xi$  is super-horizontal as well. In fact:
\begin{lem}\cite{BG}
\emph{ The map $Du_\xi:T^{1,0}U_\xi\to T^{1,0}\Omega_\xi\subset T^{1,0}U_\xi$  is given by:
  $$[\lambda^j\eta]\mapsto \left\{\begin{array}{ll}
        0 & \mbox{if $j>0$} \\
         \mbox{$[\eta]$}  &  \mbox{if $j=0$}
        \end{array}\right.. $$}
  \end{lem}

\section{Factorizations of algebraic loops}\label{xixi}
For each integer $k\geq 0$, define
$$\Omega^kG_2=\big\{\gamma\in \Omega_{\mathrm{alg}}G_2:\,\gamma=\sum_{i=-k}^k\lambda^i\zeta_i,\,\,\zeta_k\neq 0 \big\}.$$
 By a \emph{factorization} of a loop $\gamma\in\Omega^kG_2$ of \emph{length} $N$ and \emph{type } $(k_1,k_2,\ldots,k_N)$ we mean a sequence of loops $\beta_1,\beta_2,\ldots,\beta_{N-1},\beta_N$ such that $\gamma=\beta_1\beta_2\ldots\beta_{N-1}\beta_N$ and
 $\beta_i\in \Omega^{k_i}G_2$ for each $i=1,\ldots,N$.

Since each element $\xi\in I$ is diagonalizable and its eigenvalues are of the form $\sqrt{-1}k_i(\xi)$, with $k_i(\xi)\in \mathbb{Z}$, we have $\gamma_\xi\in\Omega^{\kappa(\xi)} G_2$, where $\kappa(\xi)=\max \{|k_i(\xi)|\}$. Moreover, if $\gamma\in U_\xi$, then $\gamma\in \Omega^{\kappa(\xi)} G_2$.

Consider $\gamma\in U_\xi$ and a sequence of elements $\xi=\xi_N,\xi_{N-1},\ldots,\xi_2,\xi_1,\xi_0=0$ in the integer lattice $I$ such that $\xi\preceq \xi_i$ for each $i=0,\ldots,N$. Set $k_i=\kappa(\xi_{i})-\kappa(\xi_{i-1})$.
Thus,  the sequence of loops
$$\gamma_N=\gamma, \ldots, \gamma_{i}=\mathcal{U}_{\xi,\xi_{i}}(\gamma_{N}), \ldots, \gamma_0=e$$
  induces a factorization $\beta_1,\ldots,\beta_N$ of $\gamma$, with $\beta_i=\gamma_{i-1}^{-1}\gamma_i$, of length $N$ and type $(k_1,\ldots,k_N)$.


\subsection{Canonical factorizations of algebraic loops in $G_2$}
Fix a non-zero element $\xi$ such that $\exp(2\pi\xi)=e$. According to the notations of Section \ref{rep}, for some choice of a maximal  torus $T$ of $G_2$ and a Weyl chamber $\mathcal{W}$, $\xi$ is given by
\begin{equation}\label{xi}
\xi=\left\{\begin{array}{ll}
        \sqrt{-1}\,k & \mbox{on $L_1$} \\
        \sqrt{-1}\,l &  \mbox{on $L_2$}\\
        \sqrt{-1}\,(k-l) & \mbox{on $\overline{L}_3$}
      \end{array}\right.,\end{equation}
where $k,l$ are non-negative integers such that $2l \leq k$. For each $\gamma\in U_\xi$ we define the \emph{canonical factorization}  of $\gamma$ as follows:
\subsubsection{Case $0<2l<k$} All the eigenspaces of $\xi$ are $1$-dimensional; consequently, there is a unique pair $(T,\mathcal{W})$ for which (\ref{xi}) holds.
Define $\eta_1,\eta_2\in I$ by:
$$\eta_1=\left\{\begin{array}{ll}
        \sqrt{-1} & \mbox{on $L_1$} \\
        0 &  \mbox{on $L_2$}\\
        \sqrt{-1} & \mbox{on $\overline{L}_3$}
      \end{array}\right.\quad\quad \eta_2=\left\{\begin{array}{ll}
        \sqrt{-1} & \mbox{on $L_1$} \\
        \sqrt{-1} &  \mbox{on $L_2$}\\
        0 & \mbox{on $\overline{L}_3$}
      \end{array}\right.. $$
We have $\xi=(k-l)\eta_1+l\eta_2$ and $\kappa(\eta_1)=\kappa(\eta_2)=1$.
\begin{lem}\label{ord}
 \emph{ Consider the following sequence of elements in the integer lattice $I$:
\begin{align}
\nonumber \xi&=\xi_k, \ldots,\xi_{k-i}=(k-l-i)\eta_1+l\eta_2,\ldots, \xi_{2l}=l\eta_1+l\eta_2,\xi_{2l-1}=l\eta_1+(l-1)\eta_2,\ldots\\ \nonumber
&\ldots, \xi_{2(l-j)}=(l-j)\eta_1+(l-j)\eta_2,\xi_{2(l-j)-1}=(l-j)\eta_1+(l-j-1)\eta_2,\ldots\\ \label{2l<k}&\ldots,\xi_2=\eta_1+\eta_2,\xi_1=\eta_1,\xi_0=0
\end{align}
where $0 \leq i\leq k-2l$ and $0\leq j\leq l-1$. Then  $\xi\preceq \xi_r$ for each $r=0,\ldots, k$. }
  \end{lem}
\begin{proof} Consider the elements $H_1,H_2\in I$ defined by (\ref{Hs}). We have:  $\eta_1=H_2$ and $\eta_2=H_1-H_2$.  Taking account (\ref{gis}) and (\ref{gs1}),  by
direct evaluation of the positive roots (\ref{roots}) at each $\xi_{i}$, one can conclude that $\mathfrak{p}^\xi\subset \mathfrak{p}^{\xi'}$, i.e. $\xi\preceq\xi'$.
\end{proof}
Hence, the sequence (\ref{2l<k}) defines a factorization of $\gamma$ of length $k$ and type $(1,1,\ldots,1)$.
Set $W=\gamma H_+$ and $W_\xi=\gamma_\xi H_+$. Take $\Psi\in\Lambda^+_{\mathrm{alg}}G^\C$ such that $W=\Psi\gamma_\xi H_+=\Psi W_\xi$. Next we describe the canonical factorization of $\gamma$
in terms of the Grassmanian model for loops groups.

First observe that
$$\gamma_\xi=\gamma_{\eta_1}^{k-l}\gamma_{\eta_2}^l=(\lambda^{-1}\pi_{D}+\pi_{(D\ominus \overline{D})^\perp}+\lambda\, \pi_{{{\overline{D}}}})^{k-l}(\lambda^{-1}\pi_{B}+\pi_{(B\ominus \overline{B})^\perp}+\lambda\, \pi_{{{\overline{B}}}})^{l}$$ and
\begin{multline} W_\xi=\lambda^{-k}A+\ldots+\lambda^{-k+l-1}A+\lambda^{-k+l}D+\ldots+\lambda^{-l-1}D+\lambda^{-l}A^a+\ldots+\lambda^{-1}A^a\\ +\overline{A^a}^\perp+
\ldots+\lambda^{l-1}\overline{A^a}^\perp+\lambda^{l}\overline{D}^\perp+\ldots+\lambda^{k-l-1}\overline{D}^\perp+\lambda^{k-l}\overline{A}^\perp+\ldots+\lambda^{k-1}\overline{A}^\perp+
\lambda^kH_+, \label{wxii} \end{multline}
where $A=\overline{L}_1$, $B=\overline{L}_1\oplus \overline{L}_2$, $D=\overline{L}_1\oplus {L}_3$ and $A^a$ is the annihilator of $A$, that is, $A^a=\overline{L}_1\oplus \overline{L}_2\oplus {L}_3.$

Consider now the sequence (\ref{2l<k}). We have $\gamma_{\xi_{k-1}}=\gamma_{\eta_1}^{k-l-1}\gamma_{\eta_2}^l$, and it is easy to check that
\begin{equation}\label{wxi}
W_{\xi_{k-1}}=\lambda(W_\xi\cap \lambda^{-k}H_+)+(W_\xi\cap\lambda^{-l}H_+)+\lambda^{-1}(W_\xi \cap \lambda^{l+1}H_+).
\end{equation}
Set $\gamma_{i}=\mathcal{U}_{\xi,\xi_{i}}(\gamma)$ and $W^{i}=\gamma_{i}H_+$. Taking account definition (\ref{uxi}), we deduce from (\ref{wxi}) that:
$$W^{k-1}=\Psi W_{\xi_{k-1}}=\lambda(W\cap \lambda^{-k}H_+)+(W\cap\lambda^{-l}H_+)+\lambda^{-1}(W\cap \lambda^{l+1}H_+).$$
Each element of the sequence of subspaces
$$W=W^k,\ldots,W^{k-i},\ldots, W^{2l},W^{2l-1},\ldots, W^{2(l-j)},W^{2(l-j)-1},\ldots,W^2,W^1,W^0=H_+$$
can be obtained out of $W$ by iterating this procedure. Explicitly:
\begin{align}
W^{k-i}&=\lambda^i\big(W\cap \lambda^{-k}H_+\big)+\big(W\cap\lambda^{-l}H_+ \big)+\lambda^{-i}\big(W\cap\lambda^{l+i}H_+\big);\label{1}\\
W^{2(l-j)}&=\lambda^{k-2l+j}\big(W\cap \lambda^{-k+j}H_+\big)+\lambda^j\big(W\cap \lambda^{-l}H_+\big)+\lambda^{-j}\big(W\cap \lambda^jH_+\big)\nonumber\\ &\quad\quad\quad\quad\quad\quad+\lambda^{-k+2l-j}\big(W\cap \lambda^{k-l}H_+\big)+\lambda^{2(l-j)}H_+;  \label{2}\\
W^{2(l-j)-1}&=\lambda^{k-2l+j}\big(W\cap \lambda^{-k+j+1}H_+\big)+ \lambda^{j+1}(W\cap \lambda^{-l}H_+)+\lambda^{-j-1}\big(W\cap\lambda^{j+1}H_+\big)\nonumber\\& \quad\quad\quad\quad\quad\quad+\lambda^{-k+2l-j}\big(W\cap \lambda^{k-l-1} H_+\big)+\lambda^{2(l-j)-1}H_+;\label{3}
\end{align}
where $0 \leq i\leq k-2l$ and $0\leq j\leq l-1$.

\subsubsection{Case  $l=\frac{k}{2}$} Define ${\eta}=\xi/l\in I$. In this case, $\kappa({\eta})=2$ and the sequence of elements
\begin{equation}\label{lk}
\xi=\xi_l,\ldots,\xi_i=(l-i){\eta},\ldots,\xi_1={\eta},\xi_0=0
\end{equation}
is such that $\xi\preceq\xi_{l-i}$ for each $i$. Hence (\ref{lk}) defines
a factorization of $\gamma$ of length $l=\frac{k}{2}$ and type $(2,2,\ldots,2)$.
The sequence of subspaces $W=W^l,\ldots,W^{l-i},\ldots,W^1$ corresponding to (\ref{lk}) is given by
\begin{equation}\label{5}
W^{l-i}=\lambda^i\big(W\cap \lambda^{-2l+i}H_+\big)+\lambda^{-i}\big(W\cap \lambda^i H_+\big)+\lambda^{2(l-i)}H_+.
\end{equation}
\subsubsection{Case $l=0$} Define $\eta=\xi/k$. Clearly $\kappa(\eta)=1$ and the sequence of elements
\begin{equation}\label{l=0}
\xi=\xi_k,\ldots,\xi_{k-i}=(k-i)\eta,\ldots,\xi_1=\eta,\xi_0=0
\end{equation}
is such that $\xi\preceq\xi_{k-i}$ for each $i$. Hence (\ref{l=0})
defines a factorization of $\gamma$ of length $k$ and type $(1,1,\ldots,1)$.
The corresponding sequence of subspaces $W=W^k,\ldots,W^{k-i},\ldots,W^1,W^0=H_+$ is given by
\begin{equation}\label{4}
W^{k-i}=\lambda^i\big(W\cap \lambda^{-k}H_+\big)+\big(W\cap\lambda^{-l}H_+ \big)+\lambda^{-i}\big(W\cap\lambda^{l+i}H_+\big).
\end{equation}

\section{Harmonic maps into a Lie group}

Let $M$ be a Riemann surface and  $\varphi:M\rightarrow G$  a map into
a compact matrix Lie group. Equip $G$ with a bi-invariant metric.
Define $\alpha=\varphi^{-1}{d}\varphi$ and let $\alpha=\alpha'+\alpha''$
be the type decomposition of $\alpha$ into $(1,0)$ and
$(0,1)$-forms. It is well known \cite{U}
 that $\varphi:M\rightarrow G $ is harmonic if and only if the loop of
$1$-forms given by
\begin{equation}
\label{flcon}
 \alpha_\lambda=\frac{1-\lambda^{-1}}{2}
\alpha'+\frac{1-\lambda}{2} \alpha''
 \end{equation}
 satisfies the Maurer-Cartan equation ${d}\alpha_\lambda + \frac{1}{2}[\alpha_\lambda\wedge \alpha_\lambda]=0$
 for each $\lambda\in S^1$.
Then, if $M$ is simply connected and $\varphi$ is harmonic, we can
integrate to obtain a map $\Phi:M\rightarrow \Omega G$ such that
$\alpha_\lambda=\Phi_\lambda^{-1}{d}\Phi_\lambda$ and $\Phi_{-1}=\varphi$. We call $\Phi$
an \textit{extended solution} associated to $\varphi$.

The harmonic map $\varphi:M\to G$ is of \emph{finite uniton number} if it admits an extended solution $\Phi:M\to \Omega_{\mathrm{alg}}G$, that is,
$\Phi=\sum_{i=s}^r\zeta_i\lambda^i$ for some $r\leq s\in\mathbb{Z}$. The minimal value of $r-s$, $r(\varphi)$, is called the \emph{uniton number} of $\varphi$. The reader should be
alert to the fact that $r(\varphi)$ does not coincide with the \emph{minimal uniton number} of $\varphi$ which is estimated in \cite{BG} for a general compact semi-simple Lie group.

\begin{thm}\cite{BG}\label{usd}
\emph{Let $\Phi:M\to \Omega_{\mathrm{alg}}G$ be an extended solution. Then there exists some $\xi\in I$, and some discrete subset $D$ of $M$, such that $\Phi(M\setminus D)\subseteq U_\xi$.}
  \end{thm}

Now, start with a smooth map $\Phi:M\setminus D\to U_\xi$ and consider $\Psi:M\setminus D \to \Lambda_{\mathrm{alg}}^+G^\C$ such that $\Phi H_+=\Psi\gamma_\xi H_+$. Clearly,
 $\Psi\gamma_\xi=\Phi b$ for some $b:M\setminus D\to \Lambda^+_{\mathrm{alg}}G^\C.$
Write
\begin{equation}\label{not}
\Psi^{-1}\Psi_z=\sum_{i\geq 0} X'_i\lambda^i,\,\,\,\,\Psi^{-1}\Psi_{\bar{z}}=\sum_{i\geq 0} X''_i\lambda^i.\end{equation}
Proposition 4.4 in  \cite{BG} establishes that   $\Phi$ is an extended solution if, and only if,
\begin{equation}\label{im}
\mathrm{Im} X'_i\subset \,\mathfrak{p}^\xi_{i+1},\,\,\,\,\mathrm{Im} X''_i\subset \mathfrak{p}^\xi_{i},
\end{equation}
where $\mathfrak{p}_i^\xi=\bigoplus_{j\leq i}\g_j^\xi.$ Taking account (\ref{h10}) and (\ref{t10}), this means that:
\begin{thm}\emph{A smooth map $\Phi:M\setminus D\to U_\xi$ is an extended solution if, and only if $\Phi$ is holomorphic and super-horizontal
  (that is, the derivative of $\Phi$ along $(1,0)$-direction takes values in $H^{(1,0)}U_\xi$).}
\end{thm}

Since  each fiber bundle morphisms $\mathcal{U}_{\xi,\xi'}:U_\xi\to U_{\xi'}$ is holomorphic and super-horizontal, we have the following  generalization of Theorem 4.11 in \cite{BG}:
\begin{prop}\label{popo}
\emph{Given an extended solution $\Phi:M\setminus D\to U_\xi$ and an element $\xi'\in I$ such that $\xi\preceq \xi'$, then
$\mathcal{U}_{\xi,\xi'}(\Phi)=\mathcal{U}_{\xi,\xi'}\circ \Phi:M\setminus D\to U_{\xi'}$ is a new extended solution.}
 \end{prop}On the other hand, since the bundle map $u_\xi:U_\xi\to \Omega_\xi$  is holomorphic and super-horizontal, we see that:
\begin{prop}\cite{BG}
\emph{If  $\Phi:M\setminus D\to U_\xi$ is an extended solution, then  $u_\xi\circ\Phi:M\setminus D\to \Omega_\xi$ is an extended solution.}
\end{prop}

The following lemma will be used later:
\begin{lem}\label{poi}
 \emph{ Let $\varphi:M\to G$ be the harmonic map $\varphi=\Phi_{-1}$. Then
  $$\varphi^{-1}\varphi_z=-2\sum_{i\geq 0}b(0){X'_i}^{i+1}b(0)^{-1},$$
  where ${X'_i}^{i+1}$ is the component of ${X'_i}$ over $\g^\xi_{i+1}$, with respect to the decomposition $\g^\C=\bigoplus \g^\xi_j$.}
\end{lem}
\begin{proof}
  Since $\Phi=\Psi\gamma_\xi b^{-1}$,
  $$\Phi^{-1}\Phi_z=b\gamma_\xi^{-1}\Psi^{-1}\Psi_z\gamma_\xi b^{-1}-b_zb^{-1}.$$
  It is clear that $b_zb^{-1}$ takes values in $\Lambda^+\g^\C$.
 Hence, taking account (\ref{not}), (\ref{im}) and Lemma \ref{adj},  the $\lambda^{-1}$-coefficient of  $\Phi^{-1}\Phi_z$ is given by:
 $$\sum_{i\geq 0}b(0){X'_i}^{i+1}b(0)^{-1}.$$ The lemma follows now from (\ref{flcon}).
\end{proof}

\subsection{Harmonic maps from the Grassmannian point of view}

 Let $W:{M} \rightarrow Gr(G)$ correspond to a smooth map $\Phi:M\to \Omega G$ under the
identification $\Omega G \cong Gr(G)$, that is $W=\Phi
H_+$.
Segal \cite{S} has observed that $\Phi$ is an extended solution if, and only if, $W$ is a
solution of equations:
\begin{align*}
W_z  \subset  {\lambda}^{-1}W,\quad   W_{\bar{z}}
  \subset W.
\end{align*}
The first condition means that $\frac{\partial s}{\partial z}(z)$
is contained in the subspace $\lambda^{-1}W(z)$ of $H$, for every
(smooth) map  $s : M\rightarrow H$ such that $s(z)\in
W(z)$, and it is equivalent to the super-horizontality of $\Phi$. The second condition is interpreted in a similar way and it is equivalent to the holomorphicity of $\Phi$.

\begin{rem}
  Consider some discrete set $D\subset M$, an element $\xi\in I$ and an extended solution  $\Phi:M\setminus D\to U_\xi$. As explained in Remark 3 of \cite{CP}, the bundle
  $W=\Phi H_+$ can be extended holomorphically to $M$, and, consequently,
$\Phi$ defines a global extended solution from  $M$ to $\Omega_{\mathrm{alg}}G$.
\end{rem}

If $\Phi:M\setminus D\to U_\xi$ is an extended solution and $W=\Phi H_+$, then $u_\xi(W)=u_\xi\circ\Phi H_+$ is given pointwise by
(\ref{popo1}) and we get holomorphic subbundles $A_i$ of the trivial bundle $\underline{\C}^n=M\times \C^n$ such that
$0\subsetneq A_{-s} \subseteq \ldots \subseteq A_{r-1} \subsetneq A_r=\underline{\C}^n.$
 The super-horizontally condition  implies that ${A_i}_z\subset A_{i+1}$.

%
%
%


\subsection{Normalization of harmonic maps}

The next two propositions will be used in Section \ref{pupu} to  estimate the uniton number of harmonic maps $M\to G_2$. The first one is a  generalization of Theorem 4.5 in \cite{BG}:

\begin{prop}\label{nor}
\emph{ Let $\Phi:M\setminus D\to U_\xi$ be an extended solution. Take $\xi'\in I$ such that $\xi\preceq {\xi'}$ and $\g_0^\xi=\g_0^{\xi'}$. Then there exists some constant loop $\gamma\in \Omega_{\mathrm{alg}}G$ such that $\gamma\Phi:M\setminus D\to U_{\xi'}$.}
\end{prop}
\begin{proof}First we claim that $ \mathfrak{p}^\xi_{i+1}\subset \mathfrak{p}^{\xi-\xi'}_i$, for all $i\geq 0$. In fact:

 Recall that $\p_i^\xi\subset \p_i^{\xi'}$ for all $i\geq 0$ is equivalent to
$$\g^\xi_j=\bigoplus_{0\leq k\leq j} \g^\xi_j\cap\g^{\xi'}_k$$
for all $j\geq 0$.
So, take  $X\in \g^\xi_j\cap \g_k^{\xi'}$, with $0\leq k\leq j\leq i+1$. Then
$$\mathrm{ad}(\xi-\xi')(X)=\sqrt{-1}(j-k)X,$$
that is, $X\in \g_{j-k}^{\xi-\xi'}$. If $j>0$, then $k>0$, since $\g_0^\xi=\g_0^{\xi'}$, and, consequently,
\begin{equation}\label{bobo1}
  \g_j^\xi\subset \bigoplus_{0\leq r< j}\g_r^{\xi-\xi'}\subset \p_i^{\xi-\xi'}.
\end{equation}
 If $j=0$, then $k=0$ and $X\in \g_0^{\xi-\xi'}\subset \p_i^{\xi-\xi'},$ that is,
 \begin{equation}\label{bobo2}
   \g_0^\xi\subset \p_i^{\xi-\xi'}.
 \end{equation}
On the other hand, by taking the conjugate of (\ref{bobo1}), we see that
\begin{equation}\label{bobo3}
 \g_{-j}^\xi\subset \bigoplus_{0\leq r< j}\g_{-r}^{\xi-\xi'}\subset \p_i^{\xi-\xi'}.
\end{equation}
Finally, from (\ref{bobo1}), (\ref{bobo2}) and  (\ref{bobo3}), we conclude that
$$\p_{i+1}^\xi=\bigoplus_{j\leq i}\g^\xi_j\subset \p_i^{\xi-\xi'}.$$

Now, write $\Phi H_+=\Psi\gamma_\xi H_+$, with $\Psi:M\setminus D\to \Lambda^+_{\mathrm{alg}}G$. Since $\Phi$ is super-horizontal, $\Psi^{-1}\Psi_z$ takes values in
 $$ \Lambda^+_{\mathrm{alg}}\g^\C\cap \lambda^{-1} \gamma_\xi \Lambda^+\g^\C \gamma_\xi^{-1}=  \bigoplus_{0\leq i}\lambda^i\mathfrak{p}^\xi_{i+1}\subset \bigoplus_{0\leq i }\lambda^i\mathfrak{p}^{\xi-\xi'}_{i}= \Lambda^+_{\mathrm{alg}}\g^\C\cap \gamma_{\xi-\xi'} \Lambda^+\g^\C \gamma_{\xi-\xi'}^{-1},$$
  which means that $\mathcal{U}_{\xi,\xi-\xi'}(\Phi)$ (clearly, $\xi\preceq\xi-\xi'$) is anti-holomorphic. Since $\mathcal{U}_{\xi,\xi-\xi'}(\Phi)$ is also holomorphic, we conclude that it is constant.
  Hence $\Psi\gamma_{\xi-\xi'}=\gamma^{-1} b$ for some constant loop $\gamma^{-1}\in \Omega_{\mathrm{alg}}G$ and some $b:M\to \Lambda^+_{\mathrm{alg}}G$. Then
  $$\Phi H_+=\Psi\gamma_\xi H_+=\Psi\gamma_{\xi-\xi'}\gamma_{\xi'}H_+= \gamma^{-1} b\gamma_{\xi'}H_+,$$
  which implies that $\gamma\Phi:M\to U_{\xi'}$.

\end{proof}

\begin{prop}\label{norm2}
 \emph{Let  $\Phi:M\setminus D\to U_\xi$ be an extended solution. Write $\Phi H_+=\Psi\gamma_\xi H_+$, with $\Psi:M\setminus D\to \Lambda^+_{\mathrm{alg}}G$ and   $\Psi^{-1}\Psi_z=\sum_{i\geq 0} X'_i\lambda^i$.
 Take $\xi '\in I$, with $\xi\preceq \xi'$,  such that: $\mathrm{Im}{X_0'}^1\subset \g_0^{\xi '}$, where  ${X'_0}^1$ is the component of $X'_0$ over $\g^\xi_1$; and (for $j>1$)
 \begin{equation}\label{toing}
 \g_j^\xi\subset \bigoplus_{0\leq i<j}\g^{\xi'}_i.\end{equation}
  Then $\mathcal{U}_{\xi,\xi'}(\Phi):M\setminus D\to U_{\xi'}$ is constant. Consequently, there exists some constant loop $\gamma\in \Omega_{\mathrm{alg}}G$ such that $\gamma\Phi:M\setminus D\to U_{\xi-\xi'}$. }
\end{prop}
\begin{proof}
Since $\xi\preceq \xi'$, $\mathcal{U}_{\xi,\xi'}(\Phi)$ is an extended solution. Hence, taking account Lemma \ref{poi}, we only have to check that the component of $X'_i$ over $\g^{\xi'}_{i+1}$ vanishes for all $i\geq 0$.

We know, by (\ref{im}), that $X'_0$ takes values in $\mathfrak{p}_1^\xi=\bigoplus_{i\leq 1}\g_i^\xi$. On the other hand, $\g_1^{\xi'}=\bigoplus_{i\geq 1} \g^\xi_i\cap  \g^{\xi'}_1$, since $\xi\preceq \xi'$. But, by hypothesis, $\mathrm{Im}{X'}_0^1\subset \g_0^{\xi '}$. Hence, the component of $X'_0$ over $\g_1^{\xi'}$ vanishes.

For all  $i\geq 1$, we have $\g_i^{\xi'}=\bigoplus_{j\geq i} \g^\xi_j\cap  \g^{\xi'}_i$, since $\xi\preceq \xi'$. Hence, if (\ref{toing}) holds, it follows that  $\g_i^{\xi'}=\bigoplus_{j> i} \g^\xi_j\cap  \g^{\xi'}_i$. Since, by (\ref{im}), $X'_i$ takes values in $\mathfrak{p}^\xi_{i+1}=\bigoplus_{j\leq i+1}\g_j^\xi$, we conclude that
 the component of $X'_i$ over $\g^{\xi'}_{i+1}$ vanishes for all $i\geq 1$.

The argument used in the proof of Proposition \ref{nor} can also be applied here to prove the existence of a constant loop $\gamma\in \Omega G$ such that  $\gamma \Phi:M\setminus D \to U_{\xi-\xi'}$.
\end{proof}

\subsection{Harmonic maps into inner $G$-symmetric spaces}

Let $G$ be a compact (connected) Lie group. It is well known
\cite{BG} that each connected component of $\sqrt{e}=\{g\in
G\,:\,\,g^2=e\}$ is a compact inner symmetric space. Moreover,
the embedding of each component of $\sqrt{e}$ in $G$ is totally
geodesic.

Define the involution $\mathcal{I}:\Omega G  \rightarrow \Omega G$ by $\mathcal{I}(\gamma)(\lambda)
=\gamma(-\lambda)\gamma(-1)^{-1}.$ Write $$\Omega^\mathcal{I}
G=\{\gamma\in \Omega G :\,\mathcal{I}(\gamma)=\gamma\}$$
 for the fixed set of $\mathcal{I}$. Let $M$ be a Riemann surface and
 $\Phi:M\rightarrow \Omega^\mathcal{I} G$ an extended solution. Then
 $\varphi=\Phi_{-1}$ defines a harmonic map from $M$ into a
 connected component of $\sqrt{e}$ (cf. \cite{BG}).
Conversely, if $\varphi:M\rightarrow\sqrt{e}$ is a harmonic map,  there exists an extended
 solution $\Phi:M\rightarrow \Omega^\mathcal{I} G$ such that
 $\varphi=\Phi_{-1}$. Under the identification $\Omega G\cong  Gr(G)$,
$\mathcal{I}$ induces
 an involution on $Gr(G)$, that we shall also denote by $\mathcal{I}$,
 and $\Omega^\mathcal{I} G$ can be identified with
 \bdm
 Gr^\mathcal{I}(G)=\{W\in Gr(G):
 \,\mbox{if $s(\lambda)\in W$ then $s(-\lambda)\in W$}\}.
 \edm
Corresponding to the extended solution $\Phi:M\rightarrow \Omega^\mathcal{I}
G$, consider $W=\Phi H_+:M \rightarrow
Gr^\mathcal{I}(G)$. We write $ W=W^{\textrm{even}}\oplus
W^{\textrm{odd}}$, where $W^{\textrm{even}}$ and $W^{\textrm{odd}}$
are the $+1$ and $-1$ eigenspaces of $\mathcal{I}$, respectively.

Denote $U_\xi^{\mathcal{I}}=U_\xi\cap \Omega^{\mathcal{I}}G$. We have:
\begin{prop}\label{proposition}
 \emph{ If $\xi\preceq \xi'$, then $\mathcal{U}_{\xi,\xi'}(U_\xi^{\mathcal{I}})\subset U_{\xi'}^{\mathcal{I}}$.  }
\end{prop}
\begin{proof} For $\Phi\in U_\xi$, write $\Phi H_+=\Psi \gamma_\xi H_+$, with $\Psi\in \Lambda^+_{\mathrm{alg}}G^\C$, and $\tilde{\Psi}(\lambda)=\Psi(-\lambda)$.
Start to observe that $\Phi\in U_\xi^{\mathcal{I}}$ if, and only if,
$\tilde{\Psi}=\Psi\gamma_\xi b\gamma_\xi^{-1}$
for some $b\in \Lambda^+_{\mathrm{alg}}G^\C$.

Now, suppose that $\Phi\in U_\xi^{\mathcal{I}}$. Then
$$\tilde{\Psi}=\Psi\gamma_\xi b\gamma_\xi^{-1}=\Psi\gamma_{\xi'}\gamma_{\xi'}^{-1}\gamma_\xi b\gamma_\xi^{-1}\gamma_{\xi'}\gamma_{\xi'}^{-1}.$$
Since $\Psi,\tilde{\Psi}\in  \Lambda^+_{\mathrm{alg}}G^\C$, we also have  $\gamma_\xi b\gamma_\xi^{-1}\in \Lambda^+_{\mathrm{alg}}G^\C$. Hence, by Lemma \ref{lema3}, $$\tilde{b}=\gamma_{\xi'}^{-1}\gamma_\xi b\gamma_\xi^{-1}\gamma_{\xi'}\in\Lambda^+_{\mathrm{alg}}G^\C,$$
and from here we conclude that  $\mathcal{U}_{\xi,\xi'}(\Phi)\in U_{\xi'}^{\mathcal{I}}$.
\end{proof}

Then, our previous results concerning factorizations and normalizations also hold for the symmetric case.

  \section{Harmonic maps of finite uniton number into $G_2$}\label{pupu}
  Let $\varphi:M\to G_2$ be an harmonic map of finite uniton number with extended solution ${\Phi}:M\to \Omega_{\mathrm{alg}}G_2$. By Theorem \ref{usd}, there exist some ${\xi}\in I$ of the form (\ref{xi})  and some finite subset $D$ of $M$ such that ${\Phi}:M\setminus D\to U_{{\xi}}$. Write, as usual, $W=\Phi H_+=\Psi\gamma_\xi H_+$, with $\Psi:M \setminus D\to \Lambda_{\mathrm{alg}}^+G_2^\C$.

  \subsection{Estimation of the uniton number and factorization formulae}

  Next we estimate the uniton number of $\varphi$ and we give factorization formulae for the correspondig \emph{normalized} extended solution in terms of the Grassmannian model. In particular, we will see that $r(\varphi)\leq 6$.

\subsubsection{Case $0<2l<k$} Set $\xi^\mathrm{n}=2\eta_1+\eta_2=H_1+H_2$. Then $\g^{\xi^\mathrm{n}}_0=\g^{\xi}_0=\mathfrak{t}^\C$ and ${\xi}\preceq\xi^\mathrm{n}$, which implies, by   Proposition \ref{nor}, that exists a constant loop $\gamma\in \Omega G_2$ such that   $\Phi^\mathrm{n}=\gamma \Phi:M\setminus D\to
  U_{\xi^\mathrm{n}}$. In terms of the Grassmannian model, $\gamma$ is obtained as follows:

  Since $\xi-\xi^\mathrm{n}=(k-l-2)\eta_1+(l-1)\eta_2$, then $\gamma_{\xi-\xi^\mathrm{n}}=\gamma_{\eta_1}^{k-l-2}\gamma_{\eta_2}^{l-1}$. Hence, taking account (\ref{wxii}), we have
\begin{multline*}
W_{\xi-\xi^\mathrm{n}}\equiv \gamma_{\xi-\xi^\mathrm{n}}H_+= \lambda^3(W_\xi\cap\lambda^{-k}H_+)+\lambda^2(W_\xi\cap\lambda^{-k+l}H_+)+\lambda(W_\xi\cap\lambda^{-l}H_+)\\  +(W_\xi\cap H_+)+\lambda^{-1}(W_\xi\cap\lambda^lH_+)+\lambda^{-2}(W_\xi\cap\lambda^{k-l}H_+)+\lambda^{k-3}H_+.\end{multline*}
But $\gamma^{-1}=\mathcal{U}_{\xi,\xi-\xi^\mathrm{n}}(\Phi)$. Then
\begin{multline*}
V\equiv \gamma^{-1}H_+=\Psi W_{\xi-\xi^\mathrm{n}}=\lambda^3(W\cap\lambda^{-k}H_+)+\lambda^2(W\cap\lambda^{-k+l}H_+)+\lambda(W\cap\lambda^{-l}H_+)\\  +(W\cap H_+)+\lambda^{-1}(W\cap\lambda^lH_+)+\lambda^{-2}(W\cap\lambda^{k-l}H_+)+\lambda^{k-3}H_+.\end{multline*}

   Now, define  $W^\mathrm{n}=\Phi^\mathrm{n}
H_+$. Clearly, we have $\lambda^{3}H_+\subset W^\mathrm{n}(z)\subset \lambda^{-3}H_+$ for each $z\in M\setminus D$, hence $r(\varphi)\leq 6$.
By applying pointwise the algebraic results of Section \ref{xixi}, the
 sequence
$$\xi^\mathrm{n}=\xi_3, \xi_2=\eta_1+\eta_2=H_1,\xi_1=\eta_1=H_2,\xi_0=0,$$  which, by Lemma \ref{ord}, satisfies $\xi^\mathrm{n}\preceq \xi_i $ for $i=0,1,2,3$,  gives a
factorization of $\Phi^\mathrm{n}$:
 $$\Phi^\mathrm{n}(z)=\Phi_1(z)\big(\Phi_1^{-1}(z)\Phi_2(z)\big)\big(\Phi_2^{-1}(z)\Phi_3(z)\big).$$
This is a factorization  of length  $3$ and type $(1,1,1)$.   By Proposition \ref{popo}, each $\Phi_i=\mathcal{U}_{\xi^\mathrm{n},\xi_i}({\Phi}^\mathrm{n}):M\setminus D \to U_{\xi_i}$ is an extended solution. In terms of the Grassmanian model, (\ref{1}), (\ref{2}) and (\ref{3}) give:
 \begin{align}
\nonumber W^3&=\Phi_3 H_+=W^\mathrm{n}\\
\label{factorization2} W^2&=\Phi_2 H_+=\lambda \big(W\cap \lambda^{-3}H_+\big)+\big(W\cap\lambda^{-1}H_+ \big)+\lambda^{-1}\big(W\cap\lambda^{2}H_+\big);\\
\label{factorization1} W^1&=\Phi_1 H_+=\lambda \big(W\cap \lambda^{-2}H_+\big)+ \lambda^{-1}\big(W\cap\lambda H_+\big)+\lambda H_+.
\end{align}

\begin{rem}The factorization associated to Proposition 4.13 in \cite{BG} is the factorization of  length 2 and  type $(1,2)$ defined by the sequence $\xi=H_1+H_2\preceq H_1 \preceq 0$. Hence, our factorization gives a refinement of this.
\end{rem}
\subsection{Case $l=\frac{k}{2}$} Set  $\xi^\mathrm{n}={\xi}/l$. Since $\g^{{\xi}^\mathrm{n}}_0=\g^{\xi}_0$ and ${{\xi}}\preceq{\xi^\mathrm{n}}$,  by Proposition \ref{nor}, there exists a constant loop $\gamma\in \Omega G_2$ such that   $\Phi^\mathrm{n}=\gamma \Phi:M\setminus D\to
  U_{\xi^\mathrm{n}}$. In terms of the Grassmannian model, $\gamma$ can be obtained as above:
  $$V\equiv\gamma^{-1}H_+=\lambda^2(W\cap\lambda^{-k}H_+)+\lambda(W\cap\lambda^{-l}H_+)+(W\cap H_+)+\lambda^{-1}(W\cap\lambda^lH_+)+\lambda^{k-2}H_+.$$

  In this case,  we have  $\lambda^2 H_+\subset W^\mathrm{n}(z)\subset \lambda^{-2}H_+$, for each $z\in M\setminus{D}$, and $r(\varphi)\leq 4$.
The canonical factorization of $\Phi^\mathrm{n}$ given by  (\ref{5})  is trivial, that is, has length  1.
\subsubsection{Case $l=0$} Set $\xi^\mathrm{n}={\xi}/k$. Again, there exists a constant loop $\gamma\in \Omega G_2$ such that   $\Phi^\mathrm{n}=\gamma \Phi:M\setminus D\to
  U_{\xi^\mathrm{n}}$. In this case,
    $$V\equiv\gamma^{-1}H_+=\lambda(W\cap\lambda^{-k}H_+)+(W\cap H_+)+\lambda^{k-1}H_+.$$
For each $z\in M\setminus{D}$,  $\lambda H_+\subset W^\mathrm{n}(z)\subset \lambda^{-1}H_+$. Consequently, $r(\varphi)\leq 2$.
The canonical factorization of $\Phi^\mathrm{n}$ given by  (\ref{4}) also has length  1.

 \subsection{Frenet frame data for harmonic maps into $G_2$}
As Guest \cite{G} has observed, any smooth map $W:M\to Gr$  corresponding to an extended solution  $\Phi:M\to \Omega^k_{\mathrm{alg}}U(n)$ is \emph{generated} by a certain holomorphic subbundle $X$
of the trivial bundle $\underline{\C}^{2kn}\simeq M\times \lambda^{-k}H_+\big/\lambda^kH_+$ by setting
\begin{equation*}
W=X+\lambda X^{(1)}+\ldots+\lambda^{2k-1}X^{(2k-1)}+\lambda^{k}H_+
\end{equation*}
where $X^{(i)}$ denotes the subbundle spanned by the local holomorphic sections of $X$ and their derivatives up to $i$.

\begin{rem}
 Recall the well known classification of harmonic maps $S^2\to \C P^{n-1}$ by Eells and Wood \cite{EW}: let $\phi:S^2\to\C P^{n-1}$ be any holomorphic map and $f$ a $\C^n$-valued meromorphic function on $S^2$ such that $\phi=\mathrm{Span}\{f\}$; let $i\in\{0,1,\ldots,n-1\}$ and define $\varphi:S^2\to\C P^{n-1}$ by
 $$\varphi=\mathrm{Span}\{f,f',\ldots, f^{(i)}\}\ominus\mathrm{Span}\{f,f',\ldots, f^{(i-1)}\};$$
 then $\varphi$ is harmonic; conversely, all harmonic maps $S^2\to \C P^{n-1}$ arise this way; in other words, every harmonic map $S^2\to\C P^{n-1}$ is an element of the Frenet frame of a rational curve.  Now, if $\{u_1,\ldots,u_r\}$ is a meromorphic spanning set of $X$, these  meromorphic function $u_i$ are analogous to the meromorphic function $f$ of Eells and Wood. For this reason, and following \cite{G}, we call  $\{u_1,\ldots,u_r\}$ a \emph{Frenet frame data} for the corresponding harmonic map.

\end{rem}

Let $\varphi:M\to G_2$ be an harmonic map of finite uniton number with extended solution $\Phi:M\setminus D\to U_\xi\subset \Omega_{\mathrm{alg}}G_2$, for some discrete subset $D$ and $\xi\in I$. Define $W=\Phi H_+$, which satisfies $\overline{W}^\perp=\lambda W$ and $W\cdot  W\subset W$.  As we have seen above, we may suppose that $\lambda^{3}H_+ \subset W\subset \lambda^{-3}H_+$.
If $X$ generates $W$, then the algebraic conditions on $W$ mean  that
\begin{align}
\label{s1} \langle \lambda^{i+1} s^{(i)}, \lambda^{-j}\overline{u^{(j)}}\rangle_H =0,\quad \quad \langle \lambda^{i+j} {s^{(i)}}\cdot u^{(j)},  \lambda^{-k-1}\overline{w^{(k)}}\rangle_H =0  \end{align}
for all $i,j,k=0,1,\ldots,5$ and all meromorphic sections $s$, $u$ and $w$ of $X$.
\begin{rem}
Since all meromorphic functions on the Riemann sphere $S^2$ are  rational functions,  we can always choose a meromorphic spanning set $\{u_1,\ldots,u_r\}$ of $X$ formed by polynomials in $z$ if $M=S^2$:
$$u_i(z)=P_i^0(\lambda)+P_i^1(\lambda)z+P_i^2(\lambda)z^2+\ldots +P_i^{n_i}(\lambda)z^{n_i}.$$
In this case, (\ref{s1}) becomes an algebraic system of quadratic and cubic equations for the coefficients in $\C^7$ of the $\lambda,\lambda^{-1}$-polynomials $P_i^j$.
On the other hand, all harmonic maps from $S^2$ into a compact Lie group have finite uniton number. Hence, all harmonic maps from $S^2$ into $G_2$ can be obtained by solving a algebraic system of quadratic and cubic equations.
\end{rem}

  Consider the holomorphic subbundles $A_i = p_i(W\cap\lambda^iH_+)$.
Next we give a description of the Frenet frame data  associated to such extended solutions.
Again, we have to distinguish three cases: $k=3$ and $l=1$; $k=2$ and $l=1$; $k=1$ and $l=0$.

\subsubsection{Case $k=3$ and $l=1$}  In this case,  $\xi=H_1+H_2$ and
\begin{equation}\label{seila1}
u_\xi(W)=\lambda^{-3}A+ \lambda^{-2}D+ \lambda^{-1}A^{a}+  \overline{A^a}^\perp+\lambda \overline{D}^\perp +\lambda^{2} \overline{A}^\perp+\lambda^3H_+,\end{equation}
where, for each $z\in M$, $A(z)=A_{-3}(z)$ is an isotropic line, $D(z)=A_{-2}(z)$ is a complex coassociative $2$-plane containing $A(z)$, and $A^{a}(z)=A_{-1}(z)$ is the annihilator of $A(z)$. 

Recall that  a subbundle of $\underline{\C}^n$ is said to be \emph{full} if it is not contained in a subspace $V\subsetneq \C^n$.
So, let us suppose that $A$ is full. In this case, $W$ is generated by a line bundle
$$X=\mathrm{Span} \{s=s_{-3}\lambda^{-3}+s_{-2}\lambda^{-2}+s_{-1}\lambda^{-1}+s_0+s_{1}\lambda+s_2\lambda^{2}\},$$
with $s_i:M\to\C^7$ meromorphic functions satisfying (\ref{s1}) and $s_{-3}$ a meromorphic section of $A$. With respect to the canonical factorization of $\Phi$ given by  (\ref{factorization2}) and  (\ref{factorization1}), the generating vector bundles of $W^2$ and $W^1$ are given, respectively, by:
$$X_2=\mathrm{Span}\{\lambda s,\lambda^2s^{(2)},\lambda^{4}s^{(5)}\},\quad X_1=\mathrm{Span}\{\lambda^2 s,\lambda^2s^{(1)},\lambda^{3}s^{(3)}\}.$$

If $A$ is not full, then $r(\varphi)\leq 4$:

\begin{lem}\label{nfull}
 \emph{ If $A$ is not full, then there are a constant loop $\gamma\in\Omega G_2$ and $\xi^{\mathrm{n}}\in I$ with $\kappa(\xi^{\mathrm{n}})\leq 2$ such that $\gamma\Phi:M\setminus D\to U_{\xi^{\mathrm{n}}}$. }
\end{lem}
\begin{proof}
  If $A$ is not full, then either $A$ is constant or $D$ is constant. Suppose first that $A$ is constant.
Consider ${X'_0}^1$ as is Proposition \ref{norm2}.  Since $A$ is constant (and, consequently, $A^{a}$ is also constant), by (\ref{ioio}) and (\ref{im}), ${X'_0}^1$ takes values in
$$\big\{\mathrm{Hom}(E_{-2},E_{-1})\oplus \mathrm{Hom}(E_{1},E_{2})\big\}\cap \g_2^\C\subset \g_0^{H_1}.$$
On the other hand, taking account (\ref{gis}),
\begin{align*}
\g^\xi_2=\g_{\alpha_1+\alpha_2}\subset \g^{H_1}_1,\quad \g^\xi_3=\g_{2\alpha_1+\alpha_2}\subset \g^{H_1}_2,\quad
\g^\xi_4=\g_{3\alpha_1+\alpha_2}\subset \g^{H_1}_3,\quad
\g^\xi_5=\g_{3\alpha_1+2\alpha_2}\subset \g^{H_1}_3.
\end{align*}
Hence, since $\xi\preceq H_1$, by  Proposition \ref{norm2}, there exists a constant loop $\gamma \in \Omega G_2$ such that $\gamma\Phi:M\setminus D\to U_{H_2}$. This $\gamma$ is given by $\gamma^{-1}=\mathcal{U}_{\xi, H_1}(\Phi)$ and
$$V\equiv \mathcal{U}_{\xi, H_1}(\Phi)H_+=\lambda (W\cap \lambda^{-3} H_+)+ (W\cap \lambda^{-1} H_+)+\lambda^{-1}(W\cap\lambda^2H_+)+\lambda^2 H_+.$$
In this case, we take $\xi^{\mathrm{n}}=H_2$, which means that $\kappa(\xi^{\mathrm{n}})= 1$.

If $D$ is constant, the same argument can be used to prove the existence of a constant loop $\gamma \in \Omega G_2$ such that $\gamma\Phi:M\setminus D\to U_{H_1}$. In this case, we take $\xi^{\mathrm{n}}=H_1$, which means that $\kappa(\xi^{\mathrm{n}})= 2$.
\end{proof}

\subsubsection{Case $k=2$ and $l=1$}
We have $\xi=H_1$ and
\begin{equation}\label{seila2}
u_\xi(W)=\lambda^{-2}A+ \lambda^{-1}A^{a}+  \overline{A^a}^\perp +\lambda \overline{A}^\perp+\lambda^2H_+.\end{equation}
where, for each $z\in M$, $A(z)=A_{-2}(z)$ is an isotropic line.

At most, we have to take twelve $\C^7$-valued meromorphic functions as follows: take four meromorphic sections of the form
$$\begin{array}{ll}
  s=s_{-2}\lambda^{-2}+s_{-1}\lambda^{-1}+s_0+s_{1}\lambda \quad & w=w_{-1}\lambda^{-1}+w_0+w_{1}\lambda \\
  u=u_{-1}\lambda^{-1}+u_0+u_{1}\lambda &  v=v_0+v_{1}\lambda
\end{array}$$
satisfying (\ref{s1}), and such that:
\begin{align*}
  A=\mathrm{Span}\{s_{-2}\},\,\,A^a&=\mathrm{Span}\{s_{-2},w_{-1},u_{-1}\},\,\, \overline{A^a}^\perp=\mathrm{Span}\{s_{-2},w_{-1},u_{-1},v_0\}.
\end{align*}
Then, $X$ is given by $X=\mathrm{Span}\{s,u,w,v\}+\lambda\overline{A}^\perp$.
\subsubsection{Case $k=1$ and $l=0$} We have $\xi=H_2$ and
$u_\xi(W)= \lambda^{-1}D+  \overline{D}^\perp +\lambda H_+$,
where, for each $z\in M$, $D(z)$ is a complex coassociative $2$-plane. In this case, $W$ can be obtained from four $\C^7$-valued meromorphic function as follows: take two meromorphic sections
$s=s_{-1}\lambda^{-1}+s_0$ and $w=w_{-1}\lambda^{-1}+w_0$, satisfying (\ref{s1}), such that $D$ is generated by $s_{-1}$ and $w_{-1}$; then
$X=\mathrm{Span} \{s,w\} + \overline{D}^\perp.$

\subsection{Harmonic maps into the Grassmannian of $3$-associative planes}

The exceptional Lie group $G_2$ acts transitively on the the Grassamannian of $3$-associative planes $Gr^a_3(\mathrm{Im}\,\mathbb{O})$ with isotropy
subgroup at a fixed point isomorphic to $SO(4)$. This is an
  inner symmetric space with totally geodesic embedding
  $\iota :Gr^a_3(\mathrm{Im}\,\mathbb{O})\to G_2$ given by $\iota (\mathcal{A}) =\pi_\mathcal{A} - \pi_\mathcal{A}^\perp$. Hence, harmonic
  maps into such inner symmetric spaces can be viewed as special
  harmonic maps into $G_2$.

Consider an harmonic map
$\varphi:M\to Gr^a_3(\mathrm{Im}\,\mathbb{O})$ of finite uniton number with extended solution $\Phi$ and set $W=\Phi H_+:M\to Gr_\mathcal{I}(G_2).$ If
$W(z)\ominus\lambda W(z)$ is identified with $\C^{7}$ by evaluating
at $\lambda=1$, then the element of order $2$ in $G_2$
corresponding to $W(z)$ is given by the decomposition: \bdm
W(z)\ominus\lambda W(z) =(W(z)\ominus\lambda
W(z))^{\textrm{even}}\oplus(W(z)\ominus\lambda
  W(z))^{\textrm{odd}}.
  \edm
  Since $\iota (\mathcal{A}) =\pi_\mathcal{A} - \pi_\mathcal{A}^\perp$, we must have $\dim_\mathbb{C} (W(z)\ominus\lambda
  W(z))^{\textrm{even}}=3$. The harmonic map $\varphi$ is recovered by evaluating
  $(W(z)\ominus\lambda W(z))^{\textrm{odd}}$ at $\lambda=1$.

 The filtration of $W$ by $W\cap
\lambda^{i}H_+$ \bdm W=W\cap \lambda^{-k}H_+\supseteq \ldots
\supseteq W\cap \lambda^{k-1}H_+ \supseteq W\cap
\lambda^{k}H_+=\lambda^{k}H_+ \edm induces a splitting
 \begin{equation}
 \label{sp}
  W\ominus \lambda W=V_{-k}\oplus\ldots\oplus V_k\,,
 \end{equation}
   where
 \begin{equation*}
 V_i\cong (W\cap \lambda^{i}H_+)/((\lambda W\cap
\lambda^{i}H_+)+(W\cap \lambda^{i+1}H_+))\cong A_i/A_{i-1},
\end{equation*}
with  $A_i=p_i(W\cap \lambda^iH_+)$ and  $p_i(\sum a_j\lambda^j)=a_i$.
 The involution $\mathcal{I}$ fixes the decomposition (\ref{sp}) and acts on
$V_i$ as $(-1)^{i}$. We therefore have $\sum_{i\,\, \textrm{even}}
\dim_\C V_i=3$.

If $\Phi:M\setminus D\to U_\xi^\mathcal{I}$ and $\xi\preceq \xi'$, then, by Proposition \ref{proposition}, we have $\mathcal{U}_{\xi,\xi'}(\Phi):M\setminus D\to U_{\xi'}^\mathcal{I}$. Hence,  if $\varphi:M\to
Gr^a_3(\mathrm{Im}\,\mathbb{O})$ is an harmonic map of finite uniton number, it admits an extended solution $\Phi^\mathrm{n}:M\setminus D\to U_{\xi^\mathrm{n}}^\mathcal{I}$ such that $\lambda^3H_+\subset W^\mathrm{n}\subset \lambda^{-3}H_+$, with $W^\mathrm{n}=\Phi^\mathrm{n}H_+$. Next we give a description of Frenet frame data associated to such extended solutions.
\subsubsection{Case $k=3$ and $l=1$}  In this case,  $\xi=H_1+H_2$ and $u_\xi(W)$ given by (\ref{seila1}) corresponds to the harmonic map $\varphi_\xi:M\to Gr^a_3(\mathrm{Im}\,\mathbb{O})$ given by
$\varphi_\xi=(D\ominus A)\oplus (\overline{A^a}^\perp\ominus A^a)\oplus (\overline{A}^\perp\ominus \overline{D}^\perp).$
If $A$ is full,  $W$ is generated by a line bundle
$X=\mathrm{span} \{s=s_{-3}\lambda^{-3}+s_{-1}\lambda^{-1}+s_{1}\lambda\}.$
If $A$ is not full, then, by Lemma \ref{nfull}, $r(\varphi)\leq 4$.

\subsubsection{Case $k=2$ and $l=1$}
We have $\xi=H_1$  and
$u_\xi(W)$ given by (\ref{seila2}) corresponds to the harmonic map $\varphi_\xi:M\to Gr^a_3(\mathrm{Im}\,\mathbb{O})$ given by
$\varphi_\xi=A\oplus (\overline{A^a}^\perp\ominus A^a)\oplus \overline{A}.$
In this case, we must consider six $\C^7$-valued meromorphic functions: take
$s=s_{-2}\lambda^{-2}+s_0$, $w=w_{-1}\lambda^{-1}+w_{1}\lambda$,
  $u=u_{-1}\lambda^{-1}+u_{1}\lambda$
satisfying (\ref{s1}) and such that $A=\mathrm{Span}\{s_{-2}\}$ and $A^a=\mathrm{Span}\{s_{-2},w_{-1},u_{-1}\}.$
Then, $X$ is given by $X=\mathrm{Span}\{s,w,u\}+\overline{A^a}^\perp+\lambda\overline{A}^\perp$.

\subsubsection{Case $k=1$ and $l=0$}Since $W$ takes values in $Gr_\mathcal{I}(G_2)$,  we must have:
$W=u_{\xi}(W)=\lambda^{-1}D+ \overline{D}^\perp+\lambda H_+.$ Hence, all harmonic maps $\varphi:M\to Gr^a_3(\mathrm{Im}\,\mathbb{O})$ of finite uniton number $r(\varphi)=2$ are of the form $\varphi=\overline{D}^\perp \ominus D$ for some holomorphic subbundle $D$ of $\underline{\C}^7$ such that the fibre at $z$, $D(z)$, is a complex coassociative $2$-plane for each $z\in M$.

{\tiny{Departamento de Matem\'{a}tica, Universidade da Beira
Interior, Rua Marqu\^{e}s d'\'{A}vila e Bolama 6201-001
Covilh\~{a} - Portugal

email: ncorreia@ubi.pt, rpacheco@ubi.pt}}

\end{document}